\documentclass[12pt,a4paper]{article}

\usepackage[british]{babel}

\usepackage[T1]{fontenc}

\usepackage[utf8]{inputenc}

\usepackage{csquotes}

\usepackage[backend=biber, maxbibnames=10, maxcitenames=4, maxalphanames=4, bibencoding=utf8, style=alphabetic, isbn=false, url=false, doi=true]{biblatex}
\renewbibmacro{in:}{}
\renewbibmacro*{doi+eprint+url}{%
	\printfield{doi}%
	\newunit\newblock%
	\iftoggle{bbx:eprint}{%
		\usebibmacro{eprint}%
	}{}%
	\newunit\newblock%
	\iffieldundef{doi}{%
		\usebibmacro{url+urldate}}%
	{}%
}

\usepackage{setspace}

\usepackage{lmodern}

\usepackage[indentafter]{titlesec}
\titleformat{name=\section}{}{\thetitle.}{0.8em}{\centering\scshape}
\titleformat{name=\subsection}[runin]{}{\thetitle.}{0.5em}{\bfseries}[.]
\titleformat{name=\subsubsection}[runin]{}{\thetitle.}{0.5em}{\itshape}[.]
\titleformat{name=\paragraph,numberless}[runin]{}{}{0em}{}[.]
\titlespacing{\paragraph}{0em}{0em}{0.5em}
\titleformat{name=\subparagraph,numberless}[runin]{}{}{0em}{}[.]
\titlespacing{\subparagraph}{0em}{0em}{0.5em}

\usepackage{url}

\usepackage{mathtools}

\usepackage{amssymb}

\usepackage{amsthm}

\usepackage{mathrsfs}

\usepackage{dsfont}

\usepackage{enumitem}
\setlist[enumerate]{noitemsep, partopsep=0pt, topsep=0pt, parsep=0pt, itemsep=0pt}
\setlist[itemize]{noitemsep, partopsep=0pt, topsep=0pt, parsep=0pt, itemsep=0pt}

\usepackage{interval}

\intervalconfig{soft open fences}

\usepackage[normalem]{ulem}

\usepackage[textsize=tiny, obeyDraft]{todonotes}

\setuptodonotes{fancyline}
\usepackage[stretch=10]{microtype}

\usepackage[affil-it]{authblk}

\usepackage{xfrac}

\usepackage[hidelinks,draft=false]{hyperref}

\usepackage[nameinlink,noabbrev,capitalize]{cleveref}

\usepackage[plain]{fullpage}

\theoremstyle{plain}
\newtheorem{theorem}{Theorem}
\newtheorem*{theorem*}{Theorem}
\newtheorem{proposition}[theorem]{Proposition}
\newtheorem{corollary}[theorem]{Corollary}

\newtheorem{lemma}[theorem]{Lemma}
\newtheorem*{lemma*}{Lemma}

\theoremstyle{definition}
\newtheorem{definition}[theorem]{Definition}

\theoremstyle{remark}

\newtheorem{remark}[theorem]{Remark}
\newtheorem*{remark*}{Remark}			

\newcommand*\diff{\mathop{}\!\mathrm{d}}
\newcommand{\R}{\mathbb{R}}

\newcommand{\J}{\mathrm{J}}
\newcommand{\D}{\mathcal{D}}

\newcommand{\me}{\mathrm{e}}

\renewcommand{\L}{\mathrm{L}}
\renewcommand{\exp}{\mathrm{exp}}
\renewcommand{\epsilon}{\varepsilon}
\renewcommand{\phi}{\varphi}

\DeclareMathOperator{\Lie}{\mathrm{Lie}}
\DeclareMathOperator{\T}{\mathrm{T}}

\DeclareMathOperator{\Ima}{\mathrm{Im}}

\date{\today}
\title{\textbf{\uppercase{\large{Local non-injectivity of the exponential map at
critical points in sub-Riemannian geometry}}}}
\author{Samuël Borza\thanks{\href{mailto:sborza@sissa.it}{sborza@sissa.it}} }
\affil[]{International School for Advanced Studies (SISSA, Trieste)}
\author{Wilhelm Klingenberg\thanks{\href{mailto: wilhelm.klingenberg@dur.ac.uk}{wilhelm.klingenberg@dur.ac.uk }} }
\affil[]{Department of Mathematics, Durham University}

\begin{document}

\maketitle

\providecommand{\keywords}[1]
{
	\textbf{\textit{Keywords---}} #1
}

\providecommand{\msc}[1]
{
	\textbf{\textit{MSC (2020)---}} #1
}

\begin{abstract} \noindent
	 We prove that the sub-Riemannian exponential map is not injective in any neighbourhood of certain critical points. Namely that it does not behave like the injective map of reals given by $f(x) = x^3$ near its critical point $x = 0$. As a consequence, we characterise conjugate points in ideal sub-Riemannian manifolds in terms of the metric structure of the space. The proof uses the Hilbert invariant integral of the associated variational problem.
\end{abstract}
\\
\msc{53C17, 58E10, 53C22}				

\tableofcontents

\section{Introduction}

In their 1932 landmark paper, Morse and Littauer \cite{morselittauer1932} showed that the exponential map of an analytic Finsler manifold is never injective on any neighbourhood of a conjugate vector. In other words, conjugate points occur precisely when certain families  of extremals (called extremal fields) fail to cover the neighbourhood of these points in a one-to-one manner. Savage \cite{savage1943}  extended this result to the smooth case, while Warner \cite{warner1965} supplied a different proof of the same result by obtaining normal forms of the exponential map near those conjugate vectors at which Whitney's singularity theory can be applied.

Our results in the area of sub-Riemannian geometry also require fine properties of regularity and continuity of the exponential map as established in a previous work of the present authors \cite{borklin}. This allows us to define regular/singular conjugate points. As the sub-Riemannian exponential map corresponds to the projection of solutions of a Hamiltonian system on the cotangent bundle, geodesic variations may vanish of infinite order in this regime, at least in the non-analytic case. This is in marked contrast to Riemannian geometry where the Jacobi differential equation has order two. We therefore introduce the notion of order of a conjugate points and make a distinction between those of infinite and finite order. Furthermore, the proofs of Warner, Savage, and Morse and Littauer all depend on some regularity of the conjugate locus, which remains open in sub-Riemannian geometry and that we also address. Our first main result is the following.

\begin{theorem}
\label{main1}
 Let $M$ be a sub-Riemannian manifold and $p \in M$. If $\lambda_0 \in \mathrm{Dom}(\mathrm{exp}_p) \subseteq \mathrm{T}^*_p(M)$ is a strongly normal and regular conjugate covector of finite order, then the exponential map $\mathrm{exp}_p$ is not injective in any neighbourhood of $\lambda_0$.
\end{theorem}

The content of \cref{main1} is a \textit{local} extension of the following \textit{global} known fact in (sub-)Riemannian geometry: if the structure does not admit abnormal extremals, and if $\gamma$ is a length-minimizing geodesic
with conjugate endpoints $p = \gamma(0)$ and $q = \gamma(1)$, then on any neighborhood $V \subseteq M$ of $q$ the exponential map $\exp_p|_{\exp_p^{-1}(V)} : \exp_p^{-1}(V) \to M$ is not injective (see \cite[Theorem 8.73. and Corollary 8.74.]{agrachev2020}). Thus such a result is not local, but rather global, and related with minimality: the pairs of geodesics joining
$p$ with some of the points $q \in V$ can a priori be far away at intermediate times. We do not think that the proof of the local result can be reduced to the global one, except in very specific cases like in \cref{cutlocus}, we prove the following from simple arguments (see \cref{morse-littauer-cut}). 

\begin{theorem}
    Let $M$ be an ideal sub-Riemannian manifold, and $p \in M$. If $\lambda_0 \in \mathrm{Cut}(p) \setminus \mathrm{Cut}^1(p)$, then the exponential map $\exp_p$ fails to be injective in any neighbourhood of $\lambda_0 \in \mathrm{T}^*_p(M).$
\end{theorem}

In the statement above, $\mathrm{Cut}(p)$ denotes the cotangent cut locus, i.e. the set of initial covectors corresponding to geodesics that are minimising up to time $t = 1$ but not minimising up to time $t = 1 + \epsilon$, for any $\epsilon > 0$. The subset $\mathrm{Cut}^1(p)$ of $\mathrm{Cut}(p)$ consists of the covectors $\lambda_0$ for which there exists another $\lambda_0' \in \mathrm{Cut}(p)$ such that $\exp_p(\lambda_0) = \exp_p(\lambda_0')$.

Already in Riemannian geometry, it does not seem possible to obtain the local result as expressed in \cref{main1} in full generality just as a variation on the theme (see \cref{finalremarks}). These two results appear completely different nature for at least one reason: the proof of the non-local injectivity of the exponential map relies heavily on the regularity of the conjugate locus in the neighbourhood of a critical point. Indeed, in the neighbourhood of a generic critical point, the conjugate locus is a smooth hypersurface. We obtain the sub-Riemannian analogous of this result in \cref{locusmanifold}, which is of independent interest. Because of the Hamiltonian structure of geodesics in sub-Riemannian geometry, we must introduce a technical notion of \textit{order finiteness} of a critical point. This assumption is not necessary for the global result but we do not know if it is possible to prove \cref{locusmanifold} without it.

Let us describe the main ideas of the proof, which uses a variational argument, in the spirit of Morse and Littauer \cite{morselittauer1932}, and Savage \cite{savage1943}. Assume that the family of normal extremals starting at $p$ covers the conjugate point in a one-to-one manner. This premise allows to introduce two versions of the Poincaré-Cartan one-form: the smooth and exact one-form $\eta^*_p$ on the space of initial conditions, and $\eta_p$ on the manifold itself which is only continuous due to the presence of a conjugate point. Both forms induce line integrals $I^*$ and $I$ respectively, reminiscent of the Hilbert invariant integral in the classical theory of the calculus of variations. The line integral $I^*$ is path-independent and we then claim that $I$ is also path-independent. This fact is shown by proving that the sub-Riemannian conjugate locus is a hypersurface in the neighbourhood of a regular singularity of finite order, and by using an approximation argument which is available because of Sard's theorem. From the invariance of $I^*$, we finally conclude that the geodesic joining $p$ to the conjugate point must be length-minimising, reaching a contradiction.

\cref{main1} contributes to the general theory of conjugate points in sub-Riemannian geometry but also advances in metric geometry. In \cite{sormani} Shankar and Sormani introduced a priori different \textit{synthetic} notions of conjugate points, i.e. definitions that only depend on the metric structure of the space. Our main result implies some equivalence between them in the \textit{ideal} case, that is to say when non-trivial abnormal geodesics are absent, demonstrating that if discrepancies are to be found in sub-Riemannian geometry, abnormal segments or infinite order of conjugate points must play a role.

\begin{theorem}
\label{main2}
Let $M$ be an ideal sub-Riemannian manifold, $\gamma : \interval{0}{1} \to M$ be a normal geodesic such that its initial covector $\lambda_0$ is regular conjugate and has finite order, and denote $p := \gamma(0)$ and $q := \gamma(1)$. Then, the following statements are equivalent:
\begin{enumerate}[label=\normalfont(\roman*)]
    \item $q$ is conjugate to $p$ along $\gamma$;
    \item $q$ is one-sided conjugate to $p$ along $\gamma$;
    \item $q$ is symmetrically conjugate to $p$ along $\gamma$.
\end{enumerate}
Furthermore, if $p$ and $q$ are unreachable conjugate points along $\gamma$, then $q$ is also conjugate to $p$ along $\gamma$.
\end{theorem}

The paper is organised as follows. In \cref{prelim}, we summarise notions from sub-Riemannian geometry that will be important in our argument. The family of extremals is introduced in \cref{familyextremals} while in \cref{hilbertinvariant} we define the Hilbert integrals. In \cref{conjugatelocusisregular}, we address the regularity of the sub-Riemannian conjugate locus. This result is a contribution to sub-Riemannian geometry of independent interest. Finally, we prove \cref{main1} in \cref{morselittauersavage} and discuss \cref{main2}, as well as possible future work and related open questions, in \cref{applications}. 



\section*{Data availability statement}

Data sharing not applicable to this article as no datasets were generated or analysed during the current study.

\section*{Acknowledgements}

We would like to thank Andrei Agrachev and Luca Rizzi for their interest and comments.

This project has received funding from the European Research Council (ERC) under the European Union’s Horizon 2020 research and innovation programme (grant agreement No. 945655).		

\section{Preliminaries}
\label{prelim}


We begin with a general description of \textit{sub-Riemannian geometry}, one that includes rank-varying structures. We refer the reader to \cite{agrachev2020}, for example, for details on sub-Riemannian geometry.

A sub-Riemannian structure on a manifold $M$ is given by a set of smooth global vector fields $X_1, \dots, X_m$ called the \textit{generating family}. The distribution at a point $p \in M$ is $\D_p := \mathrm{span}\{X_1(p), \dots, X_m(p)\}$. The \textit{rank} of the sub-Riemannian structure at $p \in M$ is $\mathrm{rank}(p) := \mathrm{dim}(\D_p)$. Observe that in our definition, a sub-Riemannian manifold may be rank-varying. An inner product on $\D_p$ is induced by the polarisation formula applied to the norm
\[
\|u\|^2_{\D_p} := \min \left\{ \sum_{k = 1}^{m} u^2_i \ \middle| \ \sum_{k = 1}^{m} u_i X_k(p) = u \right\}.
\]

A curve $\gamma : \interval{0}{T} \to M$ with initial value $\gamma(0) = p \in M$ is horizontal if there exists $u \in \L^{2}(\interval{0}{T}, \R^m)$, called a \textit{control}, such that $\dot{\gamma}(t) = \sum_{k = 0}^m u_k(t) X_k(\gamma(t))$ for almost every $t \in \interval{0}{T}$. In fact, from Carathéodory's existence theorem, there exists a unique maximal Lipschitz solution to the Cauchy problem
\begin{equation}
\label{cauchyprob}
\begin{cases}
	\dot{\gamma}(t) = \sum_{k = 0}^m u_k(t) X_k(\gamma(t)) \\
	\gamma(0) = p
\end{cases}
\end{equation}
for every $u \in \L^{2}(\interval{0}{T}, \R^m)$ and $p \in M$.

The \textit{sub-Riemannian length} and the \textit{sub-Riemannian energy} of $\gamma$ are defined by
	\begin{equation}
		\label{length}
		\L(\gamma) = \int_{0}^{T} \lVert \dot{\gamma} (t) \rVert_{\D_{\gamma(t)}} \diff t, \ \J(\gamma) = \dfrac{1}{2} \int_{0}^{T} \lVert \dot{\gamma} (t) \rVert^2_{\D_{\gamma(t)}} \diff t.
	\end{equation}
	
In the case where every two points can be joined by a horizontal curve, we have a well-defined distance function on $M$.

\begin{definition}
	The distance function of a sub-Riemannian manifold $M$, also called the \textit{Carnot-Carathéodory distance}, is defined by
	\[
	\diff(x, y) := \inf \{ \L(\gamma) \mid \gamma : [0,T] \to M \text{ is horizontal and } \gamma(0) = x \text{ and } \gamma(T) = y \}.
	\]
\end{definition}

In this work, we assume that the sub-Riemannian structures satisfy the \textit{Hörmander condition}, that is to say $\Lie_p(\D) = \T_p(M)$ for all $p \in M$, where $\Lie_p(\D)$ is defined as the smallest Lie algebra equipped with the Lie bracket of vector fields that contains $\D$. In that case, we also say that $\D$ is \textit{bracket-generating}. This is motivated by the well-known result from Chow and Rashevskii, see {\cite[Theorem 3.31.]{agrachev2020}}.

\begin{theorem}[Chow–Rashevskii theorem]
	Let $M$ be a sub-Riemannian manifold such that its distribution $\D$ is $\mathcal{C}^\infty$ and satisfies the Hörmander condition. Then, $(M, \diff)$ is a metric space and the manifold and metric topology of $M$ coincide.
\end{theorem}

On the space of controls $\L^{2}(\interval{0}{T}, \R^m)$, we can define a \textit{length functional}, as well as a corresponding \textit{energy functional}
\[
\mathrm L(u) := \int_{0}^{T} \| u(t) \|_{\R^m} \diff t, \ \mathrm J(u) := \dfrac{1}{2} \int_{0}^{T} \| u(t) \|^2_{\R^m} \diff t.
\]
Given an horizontal curve $\gamma : \interval{0}{T} \to M$, we define at every differentiability point of $\gamma$ the \textit{minimal control} $\overline{u}$ associated with $\gamma$
\[
\overline{u}(t) := \arg \min \left\{ \| u \|_{\R^m} \mid u \in \R^m, \dot{\gamma}(t) = \sum_{k = 0}^m u_k X_k(\gamma(t)) \right\}.
\]
The relationship with the functionals defined in \cref{length} is the following: $\L(\gamma) = \L(\overline{u})$ and $\J(\gamma) = \J(\overline{u})$ when $\overline{u}$ is the minimal control associated with $\gamma$. Through the Cauchy problem \eqref{cauchyprob}, it can be seen that finding a length minimiser for $\L$ among the horizontal curves with fixed end-points $\gamma(0) = p$ and $\gamma(T) = q$ is equivalent to finding a minimal control for $L$ for which the associated path joins $p$ and $q$. Furthermore, we have the following classical correspondence: a horizontal curve $\gamma : \interval{0}{T} \to M$ joining $p$ to $q$ is a minimiser of $\J$ if and only if it is a minimiser of $\L$ and is parametrised by constant speed.

Now that we can turn a sub-Riemannian manifold into a metric space, we would like to study the geodesics associated with its distance function. These would be horizontal curves that are locally minimising the sub-Riemannian length functional. Because of the lack of a torsion-free metric connection, we cannot have a geodesic equation through a covariant derivative. Rather, sub-Riemannian geodesic are characterised via Hamilton's equations. 

The \textit{Hamiltonian} of the sub-Riemannian structure is defined by
\begin{equation}
    \label{hamiltonmax}
    H : \T^*(M) \to \mathds{R} : \lambda \mapsto H(\lambda) := \max_{u \in \R^m} \left( \sum_{k = 1}^{m} u_k \langle \lambda, X_k(\pi(\lambda)) \rangle - \dfrac{1}{2} \sum_{k = 1}^{m} u_k ^2 \right).
\end{equation}
The Hamiltonian $H$ may written in terms of the generating family of the sub-Riemannian structure $(X_1, \dots, X_m)$ as follows
\[
H(\lambda) = \frac{1}{2}\sum_{k = 1}^m h_k(p, \lambda_0)^2, \qquad \forall \lambda \in T^*(M),
\]
where $h_k(p, \lambda_0) := \langle \lambda_0, X_k(p) \rangle$. For $p \in M$, we will also write $H_p$ for the restriction of $H$ to the cotangent space $\T_p^*(M)$.
The following result can also be seen as an application of the Maximum Principle of Pontryagin to the sub-Riemannian energy minimisation problem.

\begin{theorem}[Pontryagin's Maximum Principle]
	Let $\gamma : \interval{0}{T} \to M$ be a horizontal curve which is a length minimiser among horizontal curves, and parametrised by constant speed. Then, there exists a Lipschitz curve $\lambda(t) \in \mathrm{T}^*_{\gamma(t)}(M)$ such that one and only one of the following is satisfied:
	\begin{enumerate}
		\item[(N)] \label{hamiltoneq} $\dot{\lambda} = \overrightarrow{H}(\lambda)$, where $\overrightarrow{H}$ is the unique vector field in $\mathrm{T}^*(M)$ such that $\omega(\cdot, \overrightarrow{H}(\lambda)) = \mathrm{d}_\lambda H$ for all $\lambda \in \mathrm{T}^*(M)$ and $\omega$ denotes the canonical symplectic form on the cotangent bundle $\T^*(M)$;
		\item[(A)] \label{abnormalcondi} $\omega_{\lambda(t)}(\dot{\lambda} (t), \cap_{k = 1}^n \mathrm{ker}(\mathrm{d}_{\lambda(t)} h_k)) = 0$ for all $t \in \interval{0}{T}$.
	\end{enumerate}
	Moreover, in case (A), we have $\lambda(t) \neq 0$ for every $t \in \interval{0}{T}$.
\end{theorem}

If $\lambda$ is a curve in $\T^*(M)$ that satisfies $(N)$ (resp. $(A)$), we will also say that $\lambda$ is a normal extremal (resp. abnormal extremal). However, the projection of an abnormal extremal to $M$ might not be locally minimising. The same terminology is used for the corresponding curve $\gamma = \pi(\lambda)$ and for the minimal control associated to $\gamma$. Here, the map $\pi : \mathrm{T}^*(M) \to M$ denotes the bundle projection. We note that an extremal could be both normal and abnormal. A normal trajectory $\gamma : \interval{0}{T} \to M$ is called \textit{strictly normal} if it is not abnormal. If, in addition, the restriction $\gamma|_{\interval{0}{s}}$ is strictly normal for every $s > 0$, we say that $\gamma$ is \textit{strongly normal}. It can be seen that $\gamma$ is strongly normal if and only if the normal geodesic $\gamma$ does not contain any abnormal segment. The projection of a normal extremal to $M$ is locally minimising, that is to say it is a geodesic (for the sub-Riemannian distance) parametrised by constant-speed. More precisely, it holds $\frac{1}{2} \| \dot{\gamma} (t) \|^2 = H(\lambda(t))$ for every $t \in \interval{0}{T}$.

Sub-Riemannian structures which do not admit any non-trivial (i.e. non-constant) abnormal geodesics (the trivial geodesic is always abnormal as soon as $\mathrm{rank}(\D_p) < \mathrm{dim} (M)$) are said to be \emph{ideal}.

The theory of ordinary differential equations proves the existence of a maximal solution to $(N)$ for every given initial condition $\lambda(0) \in \T^*(M)$. The time $t$ - flow of Hamilton's equation $(N)$ is the semigroup denoted by $\me^{t \overrightarrow{H}} : \T^*(M) \to \T^*(M)$. The restriction of the time $1$ - map of this flow to the fibre $\T^*_p(M)$, followed by projection to the base, is called the sub-Riemannian \textit{exponential map} based at $p$.
\begin{definition}
	The sub-Riemannian \textit{exponential map} at $p \in M$ is the map
	\[
	\exp_p : U_p \to M : \lambda \mapsto \pi(\me^{\overrightarrow{H}}(\lambda)),
	\]
	where $U_p \subseteq \T^*_p(M)$ is the open set of covectors such that the corresponding solution of $(N)$ is defined on the interval $\interval{0}{1}$.
\end{definition}

The sub-Riemannian exponential map $\exp_p$ is smooth. If $\lambda : \interval{0}{T} \to \T^*(M)$ is the normal extremal that satisfies the initial condition $\lambda(0) = (p, \lambda_0) \in \T^*(M)$, then the corresponding normal extremal path $\gamma(t) = \pi(\lambda(t))$ by definition satisfies $\gamma(t) = \exp_p(t \lambda_0)$ for all $t \in \interval{0}{T}$. If $M$ is complete for the Carathéodory distance, then $U_p = \T^*_p(M)$, and if in addition there are no stricly abnormal length minimisers, then the exponential map $\exp_p$ is surjective. Contrary to the Riemannian case, the sub-Riemannian exponential map is not necessarily a diffeomorphism of a small ball in $\T^*_p(M)$ onto a small geodesic ball in $M$. In fact, $\Ima (\diff_0 \exp_p) = \D_p$ and $\exp_p$ is a local diffeomorphism at $0 \in \T^*_p(M)$ if and only if $\D_p = \T_p(M)$.

\begin{definition}
The critical points of the exponential map $\mathrm{exp}_p : U_p \to M$ are called {\it conjugate covectors at $p$}. We denote by $\mathrm{Conj}(p) \subseteq U_p$ the collection of all such covectors. If $s \lambda_0 \in \mathrm{T}^*_p(M)$ is conjugate, we say that the point $q := \exp_p(s \lambda_0)$ is conjugate to $p = \exp_p(0 \cdot \lambda_0)$ along the normal geodesic $\gamma(t) := \mathrm{exp}_p(t \lambda_0),$ and that $s$ is a conjugate time. 
\end{definition}

We mention here some important properties related to conjugacy. The restriction of a normal extremal $\gamma$ to an interval $\interval{t}{t+\epsilon}$ is abnormal if and only if $\gamma(s)$ is a conjugate point to $\gamma(0)$ for all $s \in \interval{t}{t+\epsilon}$ (\cite[Theorem 8.47]{agrachev2020}). Furthermore, the set of conjugate times of $\gamma$ is discrete if $\gamma$ does not contain does not contain abnormal segments (\cite[Corollary 8.51]{agrachev2020}). We will also use the following result from {\cite[Theorem 8.61.]{agrachev2020}}.

\begin{theorem}
\label{conjptnotmin}
Let $\gamma : \interval{0}{T} \to M$ be a normal extremal that does not contain abnormal segments. If $\gamma$ has no conjugate points, then it is a local minimiser for the length on the space of admissible trajectories with the same endpoints. If $\gamma$ contains at least one conjugate point to $\gamma(0)$, then it is not a local minimiser 
on the space of admissible trajectories with the same endpoints.
\end{theorem}           

\section{Non-local injectivity of the sub-Riemannian exponential map}

\subsection{A family of normal extremals and the associated flow}
\label{familyextremals}

For the rest of this work, we fix a sub-Riemannian manifold $M$ and a point $p \in M$.

\begin{definition}
The {\it family of extremals} from $p$ in augmented space is the map given by 
\[
F_p : \mathcal{U}_p \to \mathds{R} \times M : (t, \lambda_0) \mapsto (t, \mathrm{exp}_p(t \lambda_0)),
\]
where $\mathcal{U}_p \subseteq \mathds{R} \times \mathrm{T}^*_p(M)$ is the maximal open set on which $F_p$ is well-defined.
\end{definition}

The non-local injectivity of the exponential map may be related to the non-local injectivity of the family of extremals.

\begin{definition}
We say that the family of extremals containing $p \in M$ {\it simply covers $(t, q) \in \mathds{R} \times M$} if there exists an open $\mathcal{U}' \subseteq \mathcal{U}_p$ such that $F_p$ is injective on $\mathcal{U'}$ and $(t, q) \in F_p(\mathcal{U'}) \subseteq \mathds{R} \times M$.
\end{definition}

\begin{proposition}
\label{Fsimplycovers}
The exponential map 
$\mathrm{exp}_p : U_p \to M$ is injective in a neighbourhood of a covector $\lambda_0 \in U_p \subseteq \mathrm{T}^*_p(M)$ if and only if the family of extremals $F_p$ simply covers $(1, \exp_p(\lambda_0)) \in \mathds{R} \times M$.
\end{proposition}

\begin{proof}
Suppose that $\mathrm{exp}_p$ is injective in a neighbourhood $ \lambda_0 \in U' \subseteq U_p$. Then for small $\epsilon > 0$, the restriction of $F_p$ to $\ointerval{1 - \epsilon}{1 + \epsilon} \times U'$ is also injective. Indeed, the equality
$F_p(t, \lambda_0) = F_p(t', \lambda_0')$, for some $t, t' \in \ointerval{1 - \epsilon}{1 + \epsilon}$ and $\lambda_0, \lambda_0' \in U'$, implies that $t=t'$, $\exp_p(t \lambda_0) = \exp_p(t' \lambda_0')$, and thus $\lambda_0=\lambda_0'$. The other implication is proved in a similar fashion. 
\end{proof}

Similarly, we may parametrise the lift of normal extremals.

\begin{definition}
{\it The flow of extremals} based at $p \in M$ is the map given by
\[
\Phi_p : \mathcal{V}_p \to \mathds{R} \times \mathrm{T}^*(M) : (t, \lambda_0) \mapsto (t, \mathrm{e}^{t \overrightarrow{H}}(p, \lambda_0)),
\]
where $\mathcal{V}_p \subseteq \mathds{R} \times \mathrm{T}^*_p(M)$ is the maximal open set on which $\Phi_p$ is well-defined.
\end{definition}

\begin{remark}
From the definition of $\exp_p(\lambda_0)$ and $\mathrm{e}^{t \overrightarrow{H}}(p, \lambda_0) $, we know that the domains of the maps $F_p$ and $\Phi_p$ coincide: $\mathcal{U}_p = \mathcal{V}_p$.

The maps $F_p$ and $\Phi_p$ defined in this section are analogues of the concept of a {\it family of extremals} in the classical calculus of variations. Using the standard terminology, the family of extremals considered here is \textit{central} since the extremals all start from the same point $p$.
\end{remark}

A cotangent version of Gauss' lemma reads in this context as follows, see {\cite[Proposition 8.42]{agrachev2020}}.

\begin{theorem}
\label{gausslemma}
Let $p \in M$, $t \in \mathds{R}$, and $\delta : \ointerval{-\epsilon}{\epsilon} \to \mathrm{T}^*_p(M)$ such that $t \delta(s) \in U_p$ for all $s \in \ointerval{-\epsilon}{\epsilon}$. Let $\lambda(t, s) := \mathrm{e}^{t \overrightarrow{H}}(p, \delta(s))$ and $\gamma(t, s) := \pi(\lambda(t, s)) = \mathrm{exp}_p(t \delta(s))$, 
\[
\left\langle \lambda(t, s), \frac{\mathrm{d}}{\mathrm{d}s} \gamma(t, s) \right\rangle = \frac{\mathrm{d}}{\mathrm{d}s} H(\lambda(t, s)),
\]
where $\langle \cdot, \cdot \rangle$ denotes the natural pairing of covectors and vectors.
\end{theorem}


 \subsection{Hilbert invariant integrals}
\label{hilbertinvariant}

The {\it Poincaré-Cartan one-form} on $\mathds{R} \times \mathrm{T}^*(M)$, associated to the Hamiltonian of the sub-Riemannian structure, is defined to be
\[
\theta - H \mathrm{d}t \in \Omega^1(\mathds{R} \times \mathrm{T}^*(M)),
\]
where $\theta \in \Omega^1(\mathrm{T}^*(M))$ denotes the tautological one-form of the cotangent bundle. The pullback map associated to the smooth map $\Phi_p$ will be denoted by $\Phi^*_p$.

\begin{definition}
\label{hilbertint1}
The Hilbert integral $I^*_p$, defined on an open set $\mathcal{U} \subseteq \mathcal{U}_p \subseteq \mathds{R} \times \T^*_p(M)$, is the line integral obtained by integrating the one-form $\eta^*_p := \Phi_p^*(\theta - H \mathrm{d}t) \in \Omega^1(\mathcal{U})$. In other words, if $\Gamma^*$ is a smooth parametrised curve in $\mathcal{U}$, then
	\[
	I^*_p[\Gamma^*] = \int_{\Gamma^*} \eta^*_p =\int_{\Gamma^*} \Phi_p^*(\theta - H \mathrm{d}t).
	\]
\end{definition}

Let us give precise description to the action of the form $\eta_p^*$. Fix $(t, \lambda_0) \in \mathcal{U}$ and $v \in \mathrm{T}_{(t, \lambda_0)}(\mathcal{U})$, which we identify with $\mathrm{T}_{(t, \lambda_0)}(\mathds{R} \times \mathrm{T}^*_p(M))$. As usual, we set $\lambda(t) := \mathrm{e}^{t \overrightarrow{H}}(p, \lambda_0)$ and $\gamma(t) := \mathrm{exp}_p(t \lambda_0)$. By definition, we have that
\[
(\eta^*_p)_{(t, \lambda_0)}[v] = (\theta - H \mathrm{d}t)_{(t, \lambda(t))} \left[ \mathrm{d}_{(t, \lambda_0)} \Phi_p(v) \right].
\]
If we write $v = w + s \tfrac{\partial}{\partial t}{(t, \lambda_0)}$ for unique $s \in \mathds{R}$ and $w \in \mathrm{T}_{\lambda_0}(\mathrm{T}^*_p(M))$, seen as a subspace of $\mathrm{T}_{(t, \lambda_0)}(\mathds{R} \times \mathrm{T}^*_p(M))$, we can then express $\mathrm{d}_{(t, \lambda_0)}\Phi_p(v) \in \mathrm{T}_{(t, \lambda(t))}(\mathds{R} \times \mathrm{T}^*(M))$ as
\begin{equation}
\label{dPhi}
    \mathrm{d}_{(t, \lambda_0)}\Phi_p(v) = \mathrm{d}_{(p, \lambda_0)} \mathrm{e}^{t \overrightarrow{H}} \circ \mathrm{d}_{\lambda_0} \iota_{p} (w) + s \left(\frac{\partial}{\partial t}(t, \lambda(t)) + \overrightarrow{H}(\lambda(t))\right),
\end{equation}
where $\iota_p : \mathrm{T}^*_p(M) \to \mathrm{T}^*(M)$ is the injection $\iota_p(\lambda_0) := (p, \lambda_0)$, satisfying $\pi \circ \iota_p = p$ for all $p \in M$.
Now by recalling Hamilton's equation from \cref{hamiltoneq}, we obtain
\begin{align}
	\label{formulahilb1}
	(\eta^*_p)_{(t, \lambda_0)}[v] 
	& = \left\langle \lambda(t), \mathrm{d}_{\lambda_0} \mathrm{exp}_p(t \ \cdot)[w] \right\rangle + s \big( \langle \lambda(t), \dot{\gamma}(t) \rangle - H(\lambda(t)) \big).
\end{align}

The Hilbert integral $I^*_p$ has some useful properties. When evaluated along a ray in the augmented space $\mathds{R} \times \mathrm{T}^*_p(M)$, it evaluates the length of the corresponding extremal in $M$.

\begin{proposition}
\label{hilbint1islength}
Given an open neighbourhood $\mathcal{U}$ in $\mathcal{U}_p$, a covector $\lambda_0 \in \mathrm{T}^*_p(M)$, and the curve $\Gamma^* : \interval{t_0}{t_1} \to \mathds{R} \times \T_p^*(M)$ defined by $\Gamma^*(t) = (t, \lambda_0)$, then we have
\[
I^*_p[\Gamma^*] = \mathrm{L}(\gamma|_{\interval{t_0}{t_1}}),
\]
where $\gamma(t) := \mathrm{exp}_p(t \lambda_0)$, whenever $\Gamma^*(t) \in \mathcal{U}$ and $t \lambda_0 \in U_p$ for all $t \in \interval{t_0}{t_1}$.
\end{proposition}

\begin{proof}
Since $\dot{\Gamma}^*(t) = (1, 0)$, \cref{formulahilb1} implies that
\[
I^*_p[\Gamma^*] = \int_{t_0}^{t_1} \langle \lambda(t), \dot{\gamma}(t) \rangle - H(\lambda(t)) \mathrm{d}t,
\]
where $\lambda(t) := \mathrm{e}^{t \overrightarrow{H}}(p, \lambda_0)$. By \cref{hamiltonmax}, we can write the maximised Hamiltonian as
\[
H(\lambda(t)) = \langle \lambda(t), \dot{\gamma}(t) \rangle - \frac{1}{2} |u(t)|^2, \text{ for all } t \in \interval{t_0}{t_1},
\]
where $u(t)$ is the minimal control of $\gamma(t) = \mathrm{exp}_p(t \lambda_0)$, i.e. $u_i(t) = \langle \lambda(t), X_i(\gamma(t)) \rangle$ for $i = 1, \dots, N$. Therefore,
\[
I^*_p[\Gamma^*] = \int_{t_0}^{t_1} \left[ \langle \lambda(t), \dot{\gamma}(t) \rangle - \left( \langle \lambda(t), \dot{\gamma}(t) \rangle - \frac{1}{2} |u(t)|^2 \right) \right] \mathrm{d}t = \frac{1}{2} \int_{t_0}^{t_1} |u(t)|^2 \mathrm{d}t = \mathrm{L}(\gamma|_{\interval{t_0}{t_1}}).
\]
\end{proof}

Our goal now is to establish that the Hilbert integral $I^*_p$ defines an \textit{invariant} integral, namely that is it is independent of path relative to endpoints. 

\begin{proposition}
\label{I*closed}
The one-form $\eta^*_p$ defined on an open subset $\mathcal{U}$ of $\mathcal{U}_p \subseteq \mathds{R} \times \T^*_p(M)$ is closed. Equivalently, the Hilbert integral $I^*_p$ is homotopy-invariant in $\mathcal{U}$.
\end{proposition}

\begin{proof}
We prove that $\eta_p^*$ is closed by showing that that $\mathrm{d}\eta_p^* = 0$ vanishes identically. We have
	\begin{align*}
		\mathrm{d}\eta^*_p &= \mathrm{d}\left[\Phi_p^*(\theta - H \mathrm{d}t)\right] = \Phi_p^*\left[\mathrm{d}(\theta - H \mathrm{d}t)\right] = \Phi_p^*\left[\mathrm{d}(\theta - H \mathrm{d}t)\right] \\
		&= \Phi_p^*\left[\mathrm{d}\theta - \mathrm{d}H \wedge \mathrm{d}t\right] = \Phi_p^*\left[\omega - \mathrm{d}H \wedge \mathrm{d}t\right],
	\end{align*}
	where $\omega$ denotes the Poincaré two-form on $\mathrm{T}^*(M)$. If $(t, \lambda_0) \in \mathcal{U} \subseteq \mathds{R} \times \T^*_p(M)$ and $v_1, v_2 \in \mathrm{T}_{(t, \lambda_0)} (\mathcal{U}) \simeq \mathds{R} \times \T^*_p(M) $, then
	\begin{align*}
		(\mathrm{d}\eta^*_p)_{(t, \lambda_0)}(v_1, v_2) 
		&= \left[\omega - \mathrm{d}H \wedge \mathrm{d}t\right]_{(t, \lambda(t))}(\mathrm{d}_{(t, \lambda_0)}\Phi_p (v_1), \mathrm{d}_{(t, \lambda_0)}\Phi_p(v_2)).
	\end{align*}
As previously in \cref{dPhi}, we let $v_i = w_i + s_i \frac{\partial}{\partial t}{(t, \lambda_0)}$ for unique $s_i \in \mathds{R}$ and $w_i \in \mathrm{T}_{\lambda_0}(\mathrm{T}^*_p(M))$.
Without loss of generality, we only need to treat the following two cases: $s_1 = s_2 = 0$, and $s_1 = 0$ while $s_2 \neq 0$.

In the first case, we use the invariance of the symplectic form under the Hamiltonian flow, i.e. the equality $\mathrm{e}^{t \overrightarrow{H}*}\omega = \omega$, as well as the fact that $\mathrm{d}_{(p, \lambda_0)}\pi \circ \mathrm{d}_{\lambda_0} \iota_p = 0$ to find
\begin{align*}
   (\mathrm{d}\eta^*_p)_{(t, \lambda_0)}(v_1, v_2) &= \omega_{\lambda(t)}(\mathrm{d}_{(p, \lambda_0)} \mathrm{e}^{t \overrightarrow{H}} \circ \mathrm{d}_{\lambda_0} \iota_{p} (w_1), \mathrm{d}_{(p, \lambda_0)} \mathrm{e}^{t \overrightarrow{H}} \circ \mathrm{d}_{\lambda_0} \iota_{p} (w_2)) \\
    &= (\mathrm{e}^{t \overrightarrow{H}*}\omega)_{(p, \lambda_0)}(\mathrm{d}_{\lambda_0} \iota_{p} (w_1), \mathrm{d}_{\lambda_0} \iota_{p} (w_2)) = \omega_{(p, \lambda_0)}(\mathrm{d}_{\lambda_0} \iota_{p} (w_1), \mathrm{d}_{\lambda_0} \iota_{p} (w_2)) \\
    &= (\iota_p\omega)^*_{\lambda_0}(w_1, w_2) = \iota_p^*(\mathrm{d}\theta)_{\lambda_0}(w_1, w_2) = \mathrm{d}\iota_p^*(\theta)_{\lambda_0}(w_1, w_2) = 0.
\end{align*}

In the second case, we use Hamilton's equation alongside the definition of the symplectic gradient, and we obtain
\begin{align*}
	(\mathrm{d}\eta^*_p)_{(\lambda_0, t)}(v_1, v_2) &= \omega_{\lambda(t)}\left(\mathrm{d}_{(p, \lambda_0)} \mathrm{e}^{t \overrightarrow{H}} \circ \mathrm{d}_{\lambda_0} \iota_{p} (w_1), \overrightarrow{H}(\lambda(t)) \right) \\
	& \qquad - (\mathrm{d}H \wedge \mathrm{d}t)_{(t, \lambda(t))}\left(\mathrm{d}_{(p, \lambda_0)} \mathrm{e}^{t \overrightarrow{H}} \circ \mathrm{d}_{\lambda_0} \iota_{p} (w_1), \frac{\partial}{\partial t}(t, \lambda(t))\right) \\
	& \qquad - (\mathrm{d}H \wedge \mathrm{d}t)_{(t, \lambda(t), t)}\left(\mathrm{d}_{(p, \lambda_0)} \mathrm{e}^{t \overrightarrow{H}} \circ \mathrm{d}_{\lambda_0} \iota_{p} (w_1), \overrightarrow{H}(\lambda(t)) \right) \\
	&= (\mathrm{e}^{t \overrightarrow{H}}(p, \cdot)^*\mathrm{d}H)_{\lambda_0}(w_1) - (\mathrm{e}^{t \overrightarrow{H}}(p, \cdot)^*\mathrm{d}H)_{\lambda_0}(w_1) \wedge \mathrm{d}t(\overrightarrow{H}(\lambda(t))) \\
	& \qquad + (\mathrm{d}H)_{\lambda(t)}(\overrightarrow{H}(\lambda(t))) \wedge \mathrm{d}t(\mathrm{d}_{(p, \lambda_0)} \mathrm{e}^{t \overrightarrow{H}} \circ \mathrm{d}_{\lambda_0} \iota_{p} (w_1)) \\
	& \qquad - (\mathrm{e}^{t \overrightarrow{H}}(p, \cdot)^*\mathrm{d}H)_{\lambda_0}(w_1) \wedge \mathrm{d}t\left(\frac{\partial}{\partial t}(t, \lambda(t))\right) \\
	& \qquad + (\mathrm{d}H)_{\lambda(t)}\left(\frac{\partial}{\partial t}(t, \lambda(t))\right) \wedge \mathrm{d}t(\mathrm{d}_{(p, \lambda_0)} \mathrm{e}^{t \overrightarrow{H}} \circ \mathrm{d}_{\lambda_0} \iota_{p} (w_1)) = 0,
\end{align*}
which concludes the proof.
\end{proof}


We will now complete the proof of the invariance of $I^*_p$.
\begin{proposition}
\label{hilbint1exact}
Let $\mathcal{U} \subseteq \mathds{R} \times \T^*_p(M)$ be an open neighbourhood which is convex in the $\mathds{R}$-direction (i.e. if $(t_0, \lambda_0), (t_1, \lambda_0) \in \mathcal{U}$, then $(t, \lambda_0) \in \mathcal{U}$ for all $t \in \interval{t_0}{t_1}$). Then, one-form $\eta^*_p$ defined on $\mathcal{U}$ is exact. Equivalently, the Hilbert integral $I^*_p$ is path-independent in $\mathcal{U}$.
\end{proposition}

\begin{proof}
	Let $\Gamma^*(s) = (t(s), \lambda_0(s) )$ be closed path in $\mathcal{U}$ parametrised on the interval $\interval{a}{b}$. According to the convexity hypothesis, this curve is homotopic to the curve $\Gamma_0^*(s) := (t(a), \lambda_0(s))$. It follows from \cref{I*closed}, \cref{formulahilb1} and the cotangent version Gauss' lemma as seen in \cref{gausslemma} that
	\begin{align*}
	    \int_{\Gamma^*} \eta_p^* &= \int_{\Gamma_0^*} \eta_p^* = \int_a^b \frac{\mathrm{d}}{\mathrm{d}s} \left[H\left(\mathrm{e}^{t(a)\overrightarrow{H}}(p, \lambda_0(s))\right)\right] \mathrm{d}s \\
	    &= H\left(\mathrm{e}^{t(a)\overrightarrow{H}}(p, \lambda_0(b))\right) - H\left(\mathrm{e}^{t(a)\overrightarrow{H}}(p, \lambda_0(a))\right) = 0,
	\end{align*}
	since $\lambda_0(\cdot)$ is a closed curve.
\end{proof}

When a family of extremals containing $p$ simply covers an open neighbourhood of $\mathds{R} \times M$, it can be used to define a Hilbert integral $I_p$ on a subset of $\mathds{R} \times M$.

\begin{definition}
\label{hilbertint2}
Assume that the family of extremals $F_p$ is injective on a neighbourhood $\lambda_0 \in \mathcal{U} \subseteq \mathds{R} \times \T^*_p(M)$. Then the Hilbert integral $I_p$, defined on $F_p(\mathcal{U}) \subseteq \mathds{R} \times M$, is the line integral obtained by integrating the continuous one-form
\begin{equation}
    \label{inthilb2}
    (\eta_p)_{(t, q)}(s , w) = \langle \lambda_t, w\rangle - H(q, \lambda_t) s,
\end{equation}
where $(q, \lambda_t) := \mathrm{e}^{t \overrightarrow{H}}(p, \lambda_0)$, $(t, \lambda_0) := F_p^{-1}(t, q)$, and $ s \frac{\partial}{\partial t} + w \in T_{(t,q)}(\mathds{R} \times M).$ 
\end{definition}

\begin{remark}
While $\eta^*_p \in \Omega^1(\mathcal{U})$ is smooth, the form $\eta_p \in \Omega^1(F_p(\mathcal{U}))$ needs to be only continuous. Namely the functions $(q, \lambda_t)$ in its definition depend smoothly on $t, p, \lambda_0$, but the value $\lambda_0 $ in turn depends only {\it continuously} on $(t,q). $ This is since $F_p$ is only simply covering, thereby a homeomorphism. This lack of smoothness will require a detailed analysis, as given in \cref{morselittauersavage}.
\end{remark}

In a similar way to \cref{hilbint1islength}, evaluating the Hilbert integral $I_p$ on the graph of one of the normal extremals of the family $F_p$ simply gives its length.

\begin{proposition}
    Given an open neighbourhood $\mathcal{U} \subseteq \mathcal{U}_p$ such that the family of extremals $F_p$ is injective on $\mathcal{U}$, a covector $\lambda_0 \in \mathrm{T}^*_p(M)$, and the curve $\Gamma : \interval{t_0}{t_1} \to F_p(\mathcal{U}_p)$ defined by $\Gamma(t) := (t, \exp_p(t \lambda_0))$. Then the following equality holds :
	\[
	I_p[\Gamma] = \mathrm{L}(\gamma|_{\interval{t_0}{t_1}}).
	\]
	Here $\gamma(t) := \mathrm{exp}_p(t \lambda_0)$, whenever $\Gamma(t) \in F_p(\mathcal{U})$ and $t \lambda_0 \in U_p$ for all $t \in \interval{t_0}{t_1}$.
\end{proposition}

\begin{proof}
Reasoning as in the proof of \cref{hilbint1islength} and using \cref{inthilb2}, it is seen that
\[
I_p[\Gamma] = \int_\Gamma \eta_p = \int_{t_0}^{t_1} \left(\langle \mathrm{e}^{t \overrightarrow{H}}(p, \lambda_0), \dot{\gamma}(t) \rangle - H(\mathrm{e}^{t \overrightarrow{H}}(p, \lambda_0))\right) \mathrm{d}t = \mathrm{L}(\gamma|_{\interval{t_0}{t_1}}).
\]
\end{proof}

Furthermore, the Hilbert integrals and $I^*_p$ on 
$\mathcal{U}$ and $I_p$ on $F_p(\mathcal{U})$ coincide on smooth curves (assuming that $F_p$ is injective), as shown in the following proposition.

\begin{proposition}
\label{I=I*onsmooth}
   Suppose that the family of extremals $F_p : \mathcal{U} \to \mathds{R} \times M$ is an injective map. Assume also that either $\Gamma^*$ is a smooth curve in $\mathcal{U}$ and $\Gamma := F_p \circ \Gamma^*$ is the corresponding smooth curve in $F_p(\mathcal{U})$, or assume that both $\Gamma$ is a smooth curve in $F_p(\mathcal{U})$ and the corresponding curve $\Gamma^* = F_p^{-1} \circ \Gamma$ in $\mathcal{U}$ are smooth. Then $I^*_p[\Gamma^*] = I_p[\Gamma]$.
\end{proposition}

\begin{proof}
    Let us write $\Gamma(s) := (t(s), q(s))$, as well as $\Gamma^*(s) := (t(s), \lambda_0(s))$, with $q(s) := \exp_p(t_0(s) \lambda_0(s))$, where the functions $t(\cdot), q(\cdot)$ and $\lambda_0(\cdot)$ are smooth. As before we use the notation $\lambda(s) := \mathrm{e}^{t(s)\overrightarrow{H}}(p, \lambda_0(s))$ and $\gamma(s) := \mathrm{exp}_p(t_0(s), \lambda_0(s))$.
    
	Notice first that
	\begin{align*}
		\dot{\gamma}(s) &= \frac{\mathrm{d}}{\mathrm{d}s}\left(\pi \circ \mathrm{e}^{t(s)\overrightarrow{H}}(p, \lambda_0(s))\right)= \mathrm{d}_{\lambda(s)} \pi \left[ \dot{t}(s) \overrightarrow{H}(\lambda(s)) + \mathrm{d}_{\lambda_0(s)}\mathrm{e}^{t(s)\overrightarrow{H}}(p, \cdot)(\dot{\lambda}_0(s)) \right].
	\end{align*}
	Therefore, \cref{formulahilb1} and \cref{inthilb2} yield
	\begin{align*}
		(\eta_p)_{\Gamma(s)}(\dot{\Gamma}(s)) &= \langle \lambda(s), \dot{\gamma}(s)\rangle - \dot{t}(s) H(\lambda(s))\\
		&= \langle \lambda(s), \mathrm{d}_{\lambda(s)} \pi \circ \mathrm{d}_{\lambda_0(s)}\mathrm{e}^{t(s)\overrightarrow{H}}(p, \cdot)(\dot{\lambda}_0(s)) \rangle \\
		& \qquad + \dot{t}(s) \left[ \langle \lambda(s), \mathrm{d}_{\lambda(s)} \pi \circ \overrightarrow{H}(\lambda(s))\rangle - H(\lambda(s)) \right]\\
		&= (\eta^*_p)_{\Gamma^*(t)}(\dot{\Gamma}^*(t)),
	\end{align*}
	which proves the statement.
\end{proof}

The relations between $I_p$ and $I^*_p$ given by the last two propositions are insufficient to extract the exactness of $\eta_p$. In order to progress, we will now study the nature of the sub-Riemannian conjugate locus.

 \subsection{The sub-Riemannian conjugate locus}
\label{conjugatelocusisregular}

Let us recall the regularity and continuity properties of the sub-Riemannian exponential map proven by the authors in \cite{borklin}. If a Riemannian metric $g$ on $M$ is fixed, one can consider the isomorphism
	\[
    \sharp : \mathrm{Ver}_{\lambda} \subseteq \T_{\lambda}(\T^*_p(M)) \to \mathrm{T}_{p}(M) : \xi \mapsto \xi^\sharp, \qquad p := \pi(\lambda), \ \lambda \in \mathrm{T}^*(M),
    \]
where $\xi^\sharp$ is the unique element of $\mathrm{T}_p(M)$ such that $g(\xi^\sharp, X) = \xi(X)$, for every $X \in \mathrm{T}_p(M)$, the spaces $\mathrm{Ver}_\lambda$ and $\mathrm{T}^*_p(M)$ being canonically identified. As explained in \cite[Section 3.2]{borklin}, fixing a Riemannian metric in this way is equivalent to choosing a symplectic moving frame along a normal geodesic $\gamma(t)$, and extend the scalar product $\langle \cdot, \cdot \rangle_{\gamma(t)}$ along $\gamma(t)$ to the whole manifold.

The \textit{ray} in $U_p \subseteq \T_p^*(M)$ through $\lambda_0$ is the map
\[
\mathrm{r}_{p, \lambda_0} : I_{p, \lambda_0} \to \mathrm{T}^*_p(M) : t \mapsto t \lambda_0
\]
where $I_{p, \lambda_0} \subseteq \R^+$ is the maximal interval containing 0 such that $t \lambda_0 \in U_p$ for every $t \in I_{t, \lambda_0}$. In this way, $\mathrm{\dot{r}}_{p, \lambda_0}(t) \in \T_{t \lambda_0}(\T^*_p(M))$ and identifying $\T_{t \lambda_0}(\T^*_p(M))$ with $\T^*_p(M)$ in the usual way, we have $\mathrm{\dot{r}}_{p, \lambda_0}(t) = \lambda_0$ for every $t \in I_{p, \lambda_0}$.

\begin{theorem}[Warner-regularity of the sub-Riemannian exponential map \cite{borklin}] 
	\label{regsubriemexp}
	Let $M$ be a sub-Riemannian manifold and $p \in M$. Then,
	\begin{enumerate}	\item[\textbf{(R1)}] The map $\mathrm{exp}_p$ is $\mathcal{C}^\infty$ on $U_p = \mathrm{Dom}(\exp_p)$ and, for all $\lambda_0 \in U_p \setminus H^{-1}_p(0)$ and all $t \in I_{p, \lambda_0}$, we have $\mathrm{d}_{t \lambda_0} \mathrm{exp}_p(\mathrm{\dot{r}}_{p, \lambda_0}(t)) \neq 0_{\mathrm{exp}_p(t \lambda_0)}$;
	\item[\textbf{(R2)}] 
	The map
    \[
    \mathrm{Ker}(\mathrm{d}_{\lambda_0} \mathrm{exp}_p(\lambda_0)) \to \mathrm{T}_{\mathrm{exp}_p(\lambda_0)}(M) : \xi_0 \mapsto \left(\mathrm{d}_{\lambda_0} \mathrm{e}^{t \overrightarrow{H}} [\xi_0]\right)^\sharp
    \]
    has its image $g$-perpendicular to $\mathrm{Im}(\mathrm{d}_{\lambda_0} \mathrm{exp}_p)$
	\item[\textbf{(R3)}] Let $\lambda_0 \in U_p \setminus H_p^{-1}(0)$ be a covector such that the corresponding geodesic $\gamma(t) := \exp_p(t \lambda_0)$ is strongly normal. Then, there exists a radially convex neighbourhood $\mathcal{V}$ of $\lambda_0$ such that for every ray $\mathrm{r}_{p, \overline{\lambda}_0}$ which intersects $\mathcal{V}$ that does not contain abnormal subsegments in $\mathcal{V}$, the number of singularities of $\exp_p$ (counted with multiplicities) on $\mathrm{Im}(\mathrm{r}_{p, \overline{\lambda}_0}) \cap \mathcal{V}$ is constant and equals the order of $\lambda_0$ as a singularity of $\exp_p$, i.e. $\mathrm{dim} (\mathrm{Ker}(\diff_{\lambda_0} \exp_{p}))$.
	\end{enumerate}
\end{theorem}

In view of the continuity property (R3), we will now adapt Warner's definition of singular and regular conjugate points and their order (see \cite[Section 3.]{warner1965}) to the sub-Riemannian setting.

\begin{definition}
	We say that a conjugate covector $\lambda_0 \in \T^*_p(M)$ is {\it {}regular} for 
	$\exp_p $ if there is a neighbourhood $U$ of $\lambda_0$ such that every ray of $\mathrm{T}^*_p(M)$ contains at most one covector in $U$ which is conjugate. A conjugate covector that is not regular is said to be singular. The collection of regular conjugate (resp. singular conjugate) covectors is denoted by $\mathrm{Conj}^R(p) \subseteq \T^*_p(M)$ (resp. $\mathrm{Conj}^S(p)$). We write $\mathrm{Conj}^R_{U}(p) := \mathrm{Conj}^R(p) \cap U$ (resp. $\mathrm{Conj}^S_{U}(p) := \mathrm{Conj}^S(p) \cap U$).
\end{definition}

In Riemannian geometry, the Jacobi fields along a given geodesic satisfy a second order differential equation. Consequently, the order of vanishing the function $t \mapsto \mathrm{det}(\mathrm{d}_{t v} \exp_p)$ at $t = 1$ is always finite and equal to the order of $v \in \mathrm{T}_p(M)$ as a singularity of $\exp_p$ (see e.g. \cite{morse1931}). In sub-Riemannian geometry however, it is possible that this function and all its derivatives vanishes when $t = 1$. This motivates the following definition.

\begin{definition}
\label{ordercovector}
	We say that a covector $\lambda_0 \in U_p$ has order $m \in \mathds{N} \cup \left\{+\infty\right\}$ if the map $t \mapsto \det(\mathrm{d}_{t \lambda_0} \mathrm{exp}_p)$ vanishes of order $m$ at $t = 1$.
\end{definition}

The definition above is well posed thanks to (R2) of \cref{regsubriemexp}. Indeed, the domain the map
\[
\mathrm{d}_{t \lambda_0} \mathrm{exp}_p : \mathrm{T}_{t \lambda_0}(\mathrm{T}^*_p(M)) \to \mathrm{T}_{\mathrm{exp}_p(t \lambda_0)}(M)
\]
is first identified with the space $\mathrm{T}^*_p(M)$ as usual and then we have isomorphisms
\[
\mathrm{T}^*_p(M) \cong \mathrm{Ker}(\mathrm{d}_{t \lambda_0} \mathrm{exp}_p) \oplus \mathrm{Im}(\mathrm{d}_{t \lambda_0} \mathrm{exp}_p) \cong \mathrm{T}_{\mathrm{exp}_p(t \lambda_0)}(M).
\]

The last identification dependends on the choice of a moving frame in (R2). However, the vanishing of the determinant of $\mathrm{d}_{t \lambda_0} \mathrm{exp}_p$, viewed as a map from $\mathrm{T}_{\mathrm{exp}_p(t \lambda_0)}(M)$ to itself, as well as the vanishing of its derivatives that we need for \cref{ordercovector}, are independent of this choice.



\begin{proposition}
\label{constantorder}
If $\lambda_0 \in \mathrm{Conj}^R(p)$ has finite order, then there exists a neighbourhood $U \subseteq \mathrm{T}^*_p(M)$ of $\lambda_0$ such that all vectors in $\mathrm{Conj}^R_U(p)$ have the same order.
\end{proposition}

\begin{proof}
Assume that $\lambda_0$ has order $m \geq 1$. By Malgrange's Preparation Theorem (see \cite{guillemin1973}) there exists in a neighbourhood $\ointerval{1 - \epsilon}{1 + \epsilon} \times V$ of $(1, \lambda_0) \in \R \times \mathrm{T}^*_p(M)$ a factorisation of the type
\begin{equation}
\label{malprep1}
\det (\mathrm{d}_{t \eta} \mathrm{exp}_p) = c(t, \eta) \cdot \big[(t - 1)^m + a_{m - 1}(\eta) (t - 1)^{m - 1} + \cdots + a_0(\eta)\big],
\end{equation}
where $c : \ointerval{1 - \epsilon}{1 + \epsilon} \times V \to \mathds{R}$, $a_i : V \to \mathds{R}$ are smooth functions such that $c(1, \lambda_0) \neq 0$ and $a_{m - 1}(\lambda_0) = \dots = a_{0}(\lambda_0) = 0$. Since $\lambda_0$ is regular, then for sufficiently small $\epsilon' \le \epsilon$ and $V' \subseteq V$, we may assume that all the conjugate covectors in $V'$ are regular.

If $\lambda'_0 \in \mathrm{Conj}(p)$ has order $m' \geq 1$ such that $(1, \lambda'_0) \in \mathcal{V}'$, we can use Malgrange's Preparation Theorem again at $\lambda_0'$ to find a neighbouhoord $\ointerval{1 - \epsilon'}{1 + \epsilon'} \times V' \subseteq \ointerval{1 - \epsilon}{1 + \epsilon} \times \mathcal{U}$ and smooth functions $c' : \ointerval{1 - \epsilon'}{1 + \epsilon'} \times V' \to \mathds{R}, a_0', \dots, a_{m' - 1}' : V' \to \mathds{R}$ such that
\begin{equation}
    \label{malprep2}
    \det (\mathrm{d}_{t \eta} \mathrm{exp}_p) = c'(t, \eta) \cdot ((t - 1)^{m'} + a'_{m' - 1}(\eta) (t - 1)^{m' - 1} + \cdots + a'_0(\eta)),
\end{equation}
for every $(t, \eta) \in \ointerval{1 - \epsilon'}{1 + \epsilon'} \times U'$, with $c'(1, \lambda_0') \neq 0$, $a'_0(\lambda_0') = \dots = a'_{m'-1}(\lambda_0') = 0$. 

Now (\ref{malprep1}) and (\ref{malprep2}) both hold in the non-empty neighbourhood $\ointerval{1 - \epsilon'}{1 + \epsilon'} \times U'$ and therefore $m$ must equal $m'$. Indeed, $a_0(\lambda_0') = 0$ since $\lambda_0'$ is conjugate and the limit
\[
\lim_{t \to 1} \frac{(t - 1)^m + a_{m - 1}(\lambda_0') (t - 1)^{m - 1} + \cdots + a_1(\lambda_0')(t - 1)}{(t - 1)^{m'}}
\]
must exist and tend to $c'(1, \lambda_0')/c(1, \lambda_0')$. This implies that $m = m'$.

On the other hand, if a sequence $(\lambda_k)_{k \in \mathds{N}}$ of conjugate covectors that have infinite order converges to $\lambda_0$, then $\lambda_0$ would also have an infinite order, since $t \mapsto \det (\mathrm{d}_{t \eta} \mathrm{exp}_p)$ is smooth for all $\eta \in U_p$. \\
The above proves that the the set of points of constant order is open and closed in $U_p$. 
\end{proof}

\begin{remark}
    The notion of order finiteness that we introduce in the section is related to those of ample and equiregular geodesics introduced in \cite{agrachev2018}. Indeed, if $\gamma(t) := \exp_p(t \lambda'_0)$ is ample and equiregular for all $\lambda_0'$ in a neighbourhood of $\lambda_0$ in $\mathrm{T}^*_p(M)$, then $\lambda_0 \in \mathrm{Conj}^R(p)$ must have finite order. In general, the set of $\mathcal{A} \subseteq \mathrm{T}^*(M)$ of $(p, \lambda_0)$ such that the corresponding normal geodesic is ample and equiregular is non-empty and dense (see \cite[Proposition 7.1.]{agrachev2019}).
\end{remark}


In the following theorem, we prove that in the neighbourhood of a conjugate covector that has finite order, the conjugate locus can be given a structure of submanifold.

\begin{theorem}
\label{locusmanifold}
If $\lambda_0 \in \mathrm{T}^*_p(M)$ is a regular conjugate covector that has finite order $m$, then there exists a neighbourhood $U$ of $\lambda_0$ such that $\mathrm{Conj}_U(p)$ is a submanifold of codimension one in $\mathrm{T}^*_p(M)$. Moreover, $\mathrm{T}_{\lambda_0}(\mathrm{T}^*_p(M)) = \mathrm{T}_{\lambda_0}(\mathrm{Conj}_{U}(p)) \oplus \mathrm{Im}(\mathrm{r}_{p, \lambda_0})$.
\end{theorem}

\begin{proof}
Consider the smooth function
\[
\Delta^{m - 1}_p : U_p \to \mathds{R} : \lambda \mapsto \frac{\mathrm{d}^{m - 1}}{\mathrm{d}t^{m - 1}}\Big|_{t = 1} \left( \det \mathrm{d}_{t \lambda} \mathrm{exp}_p \right).
\]
From \cref{constantorder}, we can find a neighbourhood $U \subseteq U_p$ of $\lambda_0$ such that all covectors in $\mathrm{Conj}_U(p)$ are regular and of order $m$. The inclusion $\mathrm{Conj}_U(p) \subseteq (\Delta^{m - 1}_p)^{-1}(0)$ thus holds.
The derivative of $\Delta_p^{m - 1}$ along the ray $\mathrm{r}_{p, \lambda_0}$ is non-zero. Indeed, since $\lambda_0$ is assumed to have order $m$,
\begin{align*}
    \mathrm{d}_1\mathrm{r}_{p, \lambda_0}(\Delta_p^{m - 1}) &= \frac{\mathrm{d}}{\mathrm{d}s}\Big|_{s = 1}\left(\frac{\mathrm{d}^{m - 1}}{\mathrm{d}t^{m - 1}}\Big|_{t = 1} \left( \det \mathrm{d}_{t s \lambda_0} \mathrm{exp}_p \right)\right)\\
    &= (m - 1) \frac{\mathrm{d}^{m - 1}}{\mathrm{d}t^{m - 1}}\Big|_{t = 1} \left( \det \mathrm{d}_{t \lambda_0} \mathrm{exp}_p \right) + \frac{\mathrm{d}^{m}}{\mathrm{d}t^{m}}\Big|_{t = 1} \left( \det \mathrm{d}_{t \lambda_0} \mathrm{exp}_p \right) \\
    &= \frac{\mathrm{d}^{m}}{\mathrm{d}t^{m}}\Big|_{t = 1} \left( \det \mathrm{d}_{t \lambda_0} \mathrm{exp}_p \right) \neq 0.
\end{align*}
Therefore a possibly smaller neighborhood $U$ of $\lambda_0$ in $\mathrm{T}^*_p(M)$, can be chosen to be convex and on which the radial derivative of $\Delta_p^{m - 1}$ is non-zero. The function $\Delta_p^{m - 1}$ is in particular smooth on $U$ and has a non-zero differential.

Let us now show that $(\Delta^{m - 1}_p)^{-1}(0) \subseteq \mathrm{Conj}_U(p)$. Suppose $\lambda'_0 \in U$ and $\Delta_p^{m - 1}(\lambda'_0) = 0$. The neighbourhood $U$ has the property that each ray intersecting it has exactly one conjugate covector. Thus, on the ray passing through $\lambda'_0$, there must be a (unique) conjugate covector $\lambda''_0$. Since $U$ is convex, the line joining $\lambda'_0$ to $\lambda''_0$ is also in $U$. The radial derivative of $\Delta_p^{m - 1}$ along that line is non zero while $\Delta_p^{m - 1}(\lambda'_0) = 0$, by hypothesis, and $\Delta_p^{m-1}(\lambda''_0) = 0$, since all the conjugate vectors in $U$ have the same finite order. Thus, by the mean value theorem, we must have $\lambda'_0 = \lambda''_0$, and $\lambda'_0 \in \mathrm{Conj}_{U}(p)$.

Finally, by the implicit function theorem, $(\Delta_p^{m-1})^{-1}(0) = \mathrm{Conj}_U(p)$ can be given the desired manifold structure.
\end{proof}

\begin{remark}
\label{graphsubman} Under the hypothesis of \cref{locusmanifold}, we have in particular that the set $\Pi_p = \{(t, \eta) \in \ointerval{1-\epsilon}{1+\epsilon} \times U \mid t \eta \text{ is conjugate}\} \subseteq \mathcal{U}_p \subseteq \mathds{R} \times \mathrm{T}^*_p(M)$ is also a submanifold for $\epsilon > 0$ sufficiently small.
\end{remark}

\subsection{Non-local injectivity of the sub-Riemannian exponential map}
\label{morselittauersavage}

Putting everything together, we finally prove in this section the relationship between the injectivity of the sub-Riemannian exponential map and its singularities. 

We will need the following lemma, which is a consequence of Sard's theorem. It appeared in \cite[Lemma II]{savage1943}, and we provide here the proof for the sake of completeness.

\begin{lemma}
\label{savagelemma}
Let $X$ and $Y$ be smooth manifolds with $\dim X = \dim Y$, and assume that $X$ is compact. If $f$ be a $C^1$ map from $X$ to $Y$, then almost all points of $Y$ have a finite preimage under $f$.
\end{lemma}

\begin{proof}
Let us denote by $\mathrm{Crit}(f)$ the set of critical points of $f$, i.e.
\[
\mathrm{Crit}(f) := \{x \in X \mid \mathrm{d}_xf \text{ has a non-trivial kernel}\}.
\]
Let $y \in Y$ such that $f^{-1}(y)$ is infinite. 
Since $X$ is assumed to compact, we can deduce that $f^{-1}(y)$ has at least one accumulation point $x$. We write $(x_n) \subseteq f^{-1}(y)$ for a sequence, distinct from $x$, converging to $x$.

Let $\phi$ be a coordinate chart around $x$ in $X$ and $\psi$ a coordinate chart around $y$ in $Y$. We write $f_{\phi \psi} = \psi \circ f \circ \phi^{-1}$ for the expression of the map $f$ in these coordinates. When $n$ is large enough, we can be sure that $x_n \in \mathrm{Dom}(\phi)$ and then consider the map
\[
f(t) : \interval{0}{1} \to \mathds{R}^{\mathrm{dim}Y} :t \mapsto f_{\phi\psi}(\phi(x) + t (\phi(x_n) - \phi(x))).
\]
By the mean value theorem, there exists $t_n \in \ointerval{0}{1}$ such that
\[
f_{\phi\psi}(\phi(x_n)) = f_{\phi\psi}(\phi(x)) + \mathrm{d}_{\phi(x) + t_n(\phi(x_n) - \phi(x))}f_{\phi\psi} [(\phi(x_n) - \phi(x))],
\]
and therefore, by letting $n$ tend to $+\infty$,
\[
\mathrm{d}_{\phi(x)}f_{\phi\psi} [v] = 0, \text{ where } v = \lim_{n \to +\infty} \frac{\phi(x_n) - \phi(x)}{\|\phi(x_n) - \phi(x)\|} \neq 0.
\]
Consequently, $\mathrm{d}_{x} f$ has a non-trivial kernel and $x$ is a critical point, i.e. $y = f(x) \in f(\mathrm{Crit}(f))$. By Sard's theorem, the image of $\mathrm{Crit}(f)$ under $f$ has measure zero in $Y$ and the proof is complete.
\end{proof}

The study of the regularity of the conjugate locus in the previous section allows us to prove that the continuous one-form $\eta_p$ is indeed exact.

\begin{proposition}
\label{Ipathind}
If the family of extremals $F_p$ simply covers $(1, q) \in \mathds{R} \times M,$ where $q := \exp_p(\lambda_0)$ and $\lambda_0 \in U_p \subseteq \T^*_p(M)$ is a regular conjugate vector of finite order, then Hilbert integral $I_p$ is invariant on some neighbourhood of $(1, q)$ in $\mathds{R} \times M$.
\end{proposition}

\begin{proof}
By \cref{hilbint1exact}, the Hilbert integral $I^*_p$ is path-independent on a small enough convex neighbourhood $\mathcal{U}$ of $(1, \lambda_0)$ in $\mathds{R} \times \mathrm{T}^*_p(M)$. We would like to argue that as a consequence the same property holds for $I_p$.

Let $\Gamma(s) = (t(s), q(s))$ be a smooth closed curve in $F_p(\mathcal{U})$. If the corresponding curve $\Gamma^*(s) = (t(s), \lambda_0(s)) := F_p^{-1}(\Gamma(s))$ were to be smooth, we would be able to conclude this by \cref{I=I*onsmooth}. However, $F_p$ is only a homeomorphism and thus we do not have enough regularity in general to evaluate the line integral $I^*_p$ along $\Gamma^*$. This issue is addressed by inroducing a coordinate system $\phi$ containing $\Gamma$, and by approximating in $C^0$ the loop $\phi(\Gamma)$ by a sequence of parametrised polygons $(p_n)$ in $\mathds{R}^{n + 1}$. These polygons may be chosen such that they intersect the conjugate locus only a finite number of times. Indeed, if $x \notin \Pi_p$, one can deduce from \cref{savagelemma} that almost all rays passing through $x$ in $\mathds{R}^{n +1}$ intersect $\Pi_p$ a finite number of times. Finally we may conclude by \cref{I=I*onsmooth}, letting $n$ tends to $+\infty$, that
\[
I_p[\Gamma] = \lim_{n \to +\infty} I_p[\phi^{-1}(p_n)] = \lim_{n \to +\infty} I^*_p[\phi^{-1}(p_n)] = 0,
\]
that is to say, $I_p$ is path independent.
\end{proof}

The integral $I_p$ is the analogy of what is called \textit{Hilbert invariant integral} in the classical calculus of variations. The fact that the family of extremals $F_p$ induces the invariance of the Hilbert integral $I_p$ corresponds to $F_p$ defining what is commonly called a \textit{(central) field of extremals}. The strategy now is to deduce that subsegment of normal geodesics are length-minimising if they are part of a field of extremals.

\begin{proposition}
\label{Iindmin}
Let $\epsilon > 0$ and $\gamma : \rinterval{0}{1 + \epsilon} \to M$ be a strongly normal extremal starting at $p \in M$ and with initial covector $\lambda_0 \in U_p$. Assume that $\lambda_0$ is a first conjugate covector that is regular and has finite order. If the family of extremals $F_p$ simply covers $(1, \mathrm{exp}_p(\lambda_0))$, then any subsegment of $\gamma$ is local a length-minimiser 
among all the admissible trajectories with the same endpoints.
\end{proposition}

\begin{proof}
The assumption that $\lambda_0$ is strongly normal and is a first conjugate covector implies, by adapting the proof of \cref{Fsimplycovers} with the conclusions of \cref{conjptnotmin}, that the neighbourhood on which $F_p$ is injective can be chosen such that it contains the ray $\mathrm{r}_{p, \lambda_0}$.

Consider an admissible trajectory $c : \interval{t_0}{t_1} \to M$ with control $u'$, different from $\gamma|_{\interval{t_0}{t_1}}$, but with the same endpoints ($t_0, t_1 \in \ointerval{0}{1 + \epsilon}$). We denote by $\Gamma'(t) = (t, c(t))$ the corresponding curves in $F_p(\mathcal{U})$. We also write $\lambda'_0(t)$ for the curve of initial covectors in $\mathrm{T}^*_p(M)$ such that $\Gamma'(t) = F_p(t, \lambda'_0(t))$, as well as the lift $\lambda'(t) := \mathrm{e}^{t \overrightarrow{H}}(p, \lambda'_0(t))$, as usual.

By definition of the sub-Riemannian Hamiltonian, we have that
\[
H(\lambda'(t)) \geq \langle \lambda'(t), \dot{c}(t) \rangle - \frac{1}{2} |u'(t)|^2, \text{ for almost every } t \in \interval{t_0}{t_1},
\]
as well as
\[
H(\lambda(t)) = \langle \lambda(t), \dot{\gamma}(t) \rangle - \frac{1}{2} |u(t)|^2, \text{ for all } t \in \interval{t_0}{t_1},
\]
since $\lambda(t)$ is the lift of the normal geodesic $\gamma$.

Therefore, we deduce that
\[
\frac{1}{2} \mathrm{L}(\gamma|_{\interval{t_0}{t_1}})^2 = \frac{1}{2} \int_{t_0}^{t_1} |u(t)|^2 \mathrm{d}t = \int_{t_0}^{t_1} \left( \langle \lambda(t), \dot{\gamma}(t) \rangle - H(\lambda(t)) \right) \mathrm{d}t = I_p[\Gamma] = I_p[c]
\]
since the Hilbert integral $I_p$ is path-independent by \cref{Ipathind}, and
\[
\frac{1}{2} \mathrm{L}(\gamma|_{\interval{t_0}{t_1}})^2 = \int_{t_0}^{t_1} \left( \langle \lambda'(t), \dot{c}(t) \rangle - H(\overline\lambda(t)) \right) \mathrm{d}t \leq \frac{1}{2} \int_{t_0}^{t_1} |u'(t)|^2 \mathrm{d}t = \frac{1}{2} \mathrm{L}(c)^2.
\]
\end{proof}

We can now conclude recalling from \cref{conjptnotmin} that a normal extremal that does not contain any abnormal subsegment cannot be length minimising past a conjugate point. The proof of Theorem 1 will be completed by the following Theorem.

\begin{theorem}
\label{noninjthm}
Let $M$ be a sub-Riemannian manifold and $p \in M$. If $\lambda_0 \in U_p \subseteq \mathrm{T}^*_p(M)$ is a strongly normal and regular conjugate covector of finite order, then the exponential map $\mathrm{\exp}_p : U_p \to M$ is not injective in any neighbourhood of $\lambda_0$.
\end{theorem}

\begin{proof}
Without loss of generality, we may assume that $\lambda_0$ is a first conjugate covector along the geodesic it generates. Since $U_p$ is open, the geodesic $\gamma(t) := \exp_p(t \lambda_0)$ is well defined for all $t \in \rinterval{0}{1+\epsilon}$ for some $\epsilon > 0$. If $\exp_p$ were to be injective in some neighbourhood of $\lambda_0$, then the family of extremals $F_p$ would simply cover $(1, \mathrm{exp}_p(\lambda_0))$ by \cref{Fsimplycovers}. The corresponding Hilbert integral $I_p$ would be path independent by \cref{Ipathind}. Now, \cref{Iindmin} would imply that the normal geodesic $\gamma : \rinterval{0}{1+\epsilon} \to M$ is a local length minimiser among all the admissible trajectories with the same endpoints, which would be in contradiction with \cref{conjptnotmin} since $\lambda_0$ is assumed to be conjugate.
\end{proof}

In the classical theory of calculus of variations (see for example \cite{morse1932}), the fact that an extremal is part of a field of extremals is a sufficient condition for minimality is usually deduced from studying the Weierstrass E-function (see \cite{shaw1966}). That argument is replaced in optimal control problems with the maximum principle. Let us mention two instances where we have found related uses of these techniques. In \cite{nowakowski1988}, fields of extremals are constructed in order to formulate sufficient conditions of optimality generalising Weierstrass' condition. The fact that normal extremal trajectories are locally minimising is proved by constructing a (non-central) field of extremals in \cite[Section 4.7]{agrachev2020}.

Let us finally discuss how the result \cref{noninjthm} could be extended to the whole conjugate locus. If the regular conjugate covectors of finite order is dense in the regular conjugate locus, then by density the sub-Riemannian exponential map will also fail to be injective in any neighbourhood of a regular conjugate covector of infinite order. This is trivially the case if the sub-Riemannian manifold is analytic or if any geodesic $\gamma(t) := \exp_p(t \lambda_0)$ is ample and equiregular for all $(p, \lambda_0) \in \T^*(M)$. Furthermore, if the sub-Riemannian manifold is ideal then the property (R3) from \cref{regsubriemexp} implies that regular conjugate covectors is dense in the conjugate locus, as in \cite[Theorem 3.1]{warner1965}.           

\section{Applications and final remarks}
\label{applications}

\subsection{Metric geometry}

The non-local injectivity of the sub-Riemannian exponential map implies a cotangent version of Shankar-Sormani's equivalence for synthetic notions of conjugate points in length space (see \cite{sormani}) along a strongly normal geodesic with an initial covector that is regular conjugate and that has finite order.

\begin{definition}
Let $X$ be a geodesic space and denote by $\mathcal{L}$ its length structure. If $\gamma_n, \gamma$ are minimising geodesics parametrised on $\interval{0}{1}$, we say that $\gamma_n$ converges to $\gamma$ if they converge for the metric
\[
\mathrm{d}_{\mathrm{Geo}(X)}(\gamma_1, \gamma_2) := \sup_{t \in \interval{0}{1}}|\gamma_1(t) - \gamma_2(t)| + |\mathrm{L}(\gamma_1) - \mathrm{L}(\gamma_2)|.
\]
\end{definition}

If the geodesic space in question is a sub-Riemannian manifold $M$, then it is easy to see that a sequence of normal geodesics $\gamma_n(t) := \mathrm{exp}_{p_n}(t \lambda_0^n)$ converges to $\gamma(t) := \mathrm{exp}_p(t \lambda_0)$ if and only if $p_n$ converges to $p$ in $M$ and $\lambda_0^n$ converges to $\lambda_0$ in $\mathrm{T}^*_p(M)$.

\begin{definition}
\label{defconj}
Let $X$ be a geodesic space, and a geodesic $\gamma : \interval{0}{1} \to X$ joining two points $p$ and $q$ of $X$. We say that
\begin{enumerate}[label=\normalfont(\roman*)]
    \item $q$ is \textit{one-sided conjugate to $p$} along $\gamma$ if there exists a sequence of points $(q_n)$ converging to $q$ such that for every $n$, there are two distinct geodesics $\gamma_n^1$ and $\gamma^2_n$ joining $p$ to $q_n$ and both converging to $\gamma$;
    \item $p$ and $q$ are \textit{symmetrically conjugate along $\gamma$} if there exist sequences of points $(p_n)$ converging to $p$ and $(q_n)$ converging to $q$ such that that for every $n$, there are two distinct geodesics $\gamma^1_n$ and $\gamma^2_n$ joining $p_n$ to $q_n$ and both converging to $\gamma$;
    \item $p$ and $q$ are \textit{unreachable conjugate along $\gamma$} if there are sequences of points $(q_n)$ converging to $q$ and $(p_n)$ converging to $p$ such that if $\gamma_n$ is geodesic joining $q_n$ to $p_n$ for all $n$, then the sequence $(\gamma_n)$ cannot converge to $\gamma$;
    \item $p$ and $q$ are \textit{ultimate conjugate points along $\gamma$} if they are symmetrically conjugate or unreachable conjugate along $\gamma$.
\end{enumerate}
\end{definition}

The relationship in sub-Riemannian geometry between these different definitions of conjugate points is given by the following theorem.

\begin{theorem}
\label{equconj}
Let $M$ be an ideal sub-Riemannian manifold, $\gamma : \interval{0}{1} \to M$ be a normal geodesic such that its initial covector $\lambda_0$ is regular conjugate and has finite order, and denote $p := \gamma(0)$ and $q := \gamma(1)$. Then, the following statements are equivalent:
\begin{enumerate}[label=\normalfont(\roman*)]
    \item $q$ is conjugate to $p$ along $\gamma$;
    \item $q$ is one-sided conjugate to $p$ along $\gamma$;
    \item $q$ is symmetrically conjugate to $p$ along $\gamma$.
\end{enumerate}
Furthermore, if $p$ and $q$ are unreachable conjugate points along $\gamma$, then $q$ is also conjugate to $p$ along $\gamma$.
\end{theorem}

\begin{remark}
The ideal assumption in \cref{equconj} is there to ensure that the sequences of geodesics in \cref{defconj} consist of normal extremals.
\end{remark}

\begin{proof}
If $(q_n)$ is a sequence of points converging to $q$ and if $(\gamma_n^1)$, $(\gamma^2_n)$ are two sequences of geodesics joining $p$ to $q_n$ and converging to $\gamma$, then their initial covectors $\eta^1_n$ and $\eta^2_n$ will converge to $\lambda_0$ in $\mathrm{T}^*_p(M)$. By \cref{noninjthm}, the normal extremals $\gamma^1_n$ and $\gamma_n^2$ must coincide for $n$ large enough. This proves that (i) implies (ii). It is easy to see from \cref{defconj} that (ii) immediately implies (iii).

Consider the function
\[
E : \mathrm{T}^*(M) \to M \times M : (x, \eta) \mapsto (x, \mathrm{exp}_p(\eta)),
\]
which has a differential at $(p, \lambda_0)$ that is a linear isomorphism if $q$ is not conjugate to $p$ along $\gamma$. By the inverse function theorem, there exists a neighbourhood of $(p, \lambda_0)$ on which the function $E$ is a diffeomorphism. Suppose now that $q$ is one-sided conjugate to $p$ along $\gamma$, so that there exists a sequence of points $(q_n)$ converging to $q$ and for all $n$ two distinct normal extremals $\gamma^1_n$ and $\gamma^2_n$, with initial covectors $\eta^1_n$ and $\eta^2_n$ respectively, joining $p$ to $q_n$ and converging to $\gamma$. The proof of (ii) implies (i) follows by observing that $E(p, \eta^1_n) = E(p, \eta^2_n)$ for $n$ large enough while $\eta^1_n \neq \eta^2_n$ since $\gamma_n^1$ and $\gamma^2_n$ are distinct. By the same argument, it can be seen that (iii) implies (i). The last statement about unreachable conjugate points is the same as in \cite{sormani}.
\end{proof}

\subsection{Structure of the sub-Riemannian cut locus}
\label{cutlocus}

\cref{main1} is also related to the structure of the sub-Riemannian cut locus. Let $M$ be an ideal sub-Riemannian manifold and define the cut time of $(p, \lambda_0) \in \mathrm{T}^*_p(M)$ as
\[
t_{\text{cut}}(p, \lambda_0) := \sup \{ t > 0 \mid \exp_p(\cdot \lambda_0)|_{\interval{0}{t}} \text{ is a length-minimising geodesic}\}.
\]
The (cotangent) cut locus is then
\[
\mathrm{Cut}(p) := \{ \lambda_0 \in \mathrm{T}^*_p(M) \mid t_{\text{cut}}(p, \lambda_0) = 1\}.
\]
and we denote by $\mathrm{Cut}^1(p)$ the subset of $\mathrm{Cut}(p)$ consisting of those covectors $\lambda_0$ for which there exists another $\lambda_0' \in \mathrm{Cut}(p)$ such that $\exp_p(\lambda_0) = \exp_p(\lambda_0')$ and $\lambda_0 \neq \lambda_0'$. Note that by \cite[Theorem 8.72.]{agrachev2020}, if $\lambda_0 \in \mathrm{Cut}(p) \setminus \mathrm{Cut}^1(p)$, then $\lambda_0 \in \mathrm{Conj}(p)$. In this specific case, we are able to prove the local non-injectivity property of the exponential map without extra assumptions and without invoking the regularity of the conjugate locus.

\begin{theorem}
    \label{morse-littauer-cut}
    Let $M$ be an ideal sub-Riemannian manifold, and $p \in M$. If $\lambda_0 \in \mathrm{Cut}(p) \setminus \mathrm{Cut}^1(p)$, then the exponential map $\exp_p$ fails to be injective in any neighbourhood of $\lambda_0$.
\end{theorem}

\begin{proof}
    Let $\gamma(t) := \exp_p(t \lambda_0)$ be the unique length minimising geodesic between $p$ and $q := \exp_p(\lambda_0)$. Let $u$ be its minimal control and $\lambda_1 \in \mathrm{T}^*_p(M)$ its Langrange multiplier. By \cite[Corollary 8.74.]{agrachev2020}, we find a sequence of points $q_k \in M$ converging to $q$ such that for each $k$ there are two distinct length minimisers $\gamma_k^1$ and $\gamma_k^2$ joining $p$ and $q_k$. We denote by $\lambda_{1, k}^1$ and $\lambda^2_{1, k}$ the normal Lagrange multipliers of $\gamma_k^1$ and $\gamma_k^2$ respectively, as well as $u_k^1$ and $u^2_k$ for their respective minimal control. We write $\lambda_{0, k}^1$ and $\lambda_{0, k}^2$ for their initial covectors. They satisfy $\exp_p(\lambda_{0, k}^1) = \exp_p(\lambda_{0, k}^2)$ and we would like to prove that $\lambda_{0, k}^1$ and $\lambda_{0, k}^2$ both converge to $\lambda_0$ as $k \to +\infty$.

    Modulo extraction of a subsequence, we may assume by compactness of length minimisers (\cite[Proposition 8.67.]{agrachev2020}) that $u_k^1$ and $u^2_k$ converge in the strong $L^2$ topology, as well as $\gamma_k^1$ and $\gamma_k^2$ converge uniformly to a geodesic joining $p$ and $q$. This is unique by our assumption. So, $\gamma_k^1$ and $\gamma_k^2$ converge uniformly to $\gamma$ and $u_k^1$ and $u^2_k$ converge to $u$. 

    Choose any metric $|\cdot|$ on $\mathrm{T}^*_p(M)$ that we only use to prove estimates. We want to show that $\lambda_{1, k}^1$ is convergent. Assume that it is not, this would mean that there exists a subsequence of $|\lambda_{1, k}^1|$ that diverges to $+\infty$. Now, $\eta_{1, k}^1 := \lambda_{1, k}^1/|\lambda_{1, k}^1|$ of course converges to some $\eta_{1}^1$. The Lagrange multiplier rule (\cite[Section 8.3]{agrachev2020}) for $\gamma^1_k$ implies that
    \[
    \frac{\lambda_{1, k}^1}{|\lambda_{1, k}^1| } D_{u_k} E_{p} = \frac{u_k}{|\lambda_{1, k}^1| }
    \]
    and thus, by taking the limit, $\eta_{1}^1 D_{u} E_{p} = 0$. Here, we have written $E_p$ for the endpoint map (see \cite[Section 8.1]{agrachev2020}). We obtain that $\eta_{1}^1$ is an abnormal Lagrange multiplier but this is impossible since we assumed $\gamma$ to be non-abnormal.

    Therefore, $\lambda_{1, k}^1$ must be convergent and by taking again the limit in the Lagrange multiplier rule
    \[
    \lambda_{1, k}^1 D_{u_k} E_{p} = u_k
    \]
    we obtain that its limit must a (normal) Lagrange multiplier for $\gamma$. This is necessarily unique since $\gamma$ is not abnormal and thus $\lambda_{1, k}^1 \to \lambda_1$, as well as $\lambda_{1, k}^2 \to \lambda_1$ by the same argument. This also implies that $\lambda_{0, k}^1 \to \lambda_0$ and $\lambda_{0, k}^2 \to \lambda_0$ since initial and final covectors are linked by the Hamiltonian flow.
\end{proof}

\begin{corollary}
    \label{density}
    Let $M$ be an ideal sub-Riemannian manifold, and $p \in M$. The set $\mathrm{Cut}^1(p)$ is dense in $\mathrm{Cut}(p)$.
\end{corollary}

\begin{proof}
    Let $\lambda_0 \in \mathrm{Cut}(p) \setminus \mathrm{Cut}^1(p)$ and consider a sequence of decreasing open neighbourhoods $(U_n) \subseteq \mathrm{T}^*_p(M)$ such that $\cap_n U_n = \{\lambda_0\}$. Then, by \cref{noninjthm}, the exponential map $\exp_p$ fails to be injective on each $U_n$ and so for every $n$, the intersection $\mathrm{Cut}^1(p) \cap U_n$ is non emptry. We therefore find a sequence of covectors $\lambda_n \in \mathrm{Cut}^1(p)$ such that $\lambda_n \to \lambda_0$.
\end{proof}

\cref{density} can be used to prove that the map $\mathrm{d}(p, \cdot)^2$ is smooth in a neighbourhood of $q$ if and only if there is a unique length-minimiser $\gamma$ joining $p$ to $q$ and $q$ is not conjugate to $p$ along $\gamma$. In the non-ideal case, one may bypass this type of argument to study the regularity of the sub-Riemannian squared distance (see \cite[Chapter 11]{agrachev2020}).

\subsection{Final remarks}
\label{finalremarks}
\begin{enumerate}[label=\arabic*)]
    \item \cite{klingenberg1995} gives an inconclusive argument towards non-local injectivity of the Riemannian exponential map at critical points. Namely  \cite[2.1.12 Theorem]{klingenberg1995} argues indirectly, assuming that $\exp_p$ is injective in a neighbourhood of a conjugate vector $v$ (and thus a homeomorphism by invariance of domain). A curve $b$ that has the same endpoints as the geodesic $c$ along which $p$ is conjugate to $\exp_p(v)$ is constructed so that one may use standard arguments using Jacobi fields to conclude that $\mathrm{L}(b) < \mathrm{L}(c)$. The final step consists in defining the curve $\tilde{b} := \exp_p^{-1}(b)$ in $\mathrm{T}_p(M)$ and to invoke a corollary of Gauss' lemma from \cite[1.9.2 Theorem]{klingenberg1995} to conclude that $\mathrm{L}(b) \geq \mathrm{L}(c)$. However, this very last step can only be performed if $\tilde{b}$ is (piecewise) smooth. However, $\exp_p$ needs to be only a homeomorphism. This is in fact the heart of the problem that needs to be overcome, from the orginal arguments of Morse - Littauer, and Savage (\cite{morselittauer1932}, \cite{savage1943}), to the one of Warner (\cite{warner1965}), and indeed the present work.
    \item Even if the sub-Riemannian manifold $M$ is ideal, it seems possible that a geodesic has a conjugate point of infinite order. Finding a relevant example could help to fully understand the structure of the conjugate locus. Is it possible to construct an example for which $\mathrm{Conj}(p)$ contains only conjugate covectors of infinite order?
    \item The proof of \cref{noninjthm} relies on the submanifold structure of $\mathrm{Conj}(p)$ near a regular conjugate covector of finite order. In the case of infinite order, can we still find such a manifold structure?
    \item Warner's proof of \cref{noninjthm} is different to what we have done here. At the core of Warner's proof is Whitney's singularity theorem to find the normal forms of the exponential map near a conjugate vector. What are the normal forms of the sub-Riemannian exponential map? Due to the reliance of this argument on the order of the conjugate value, there might not be a general answer to that question. Rather each class of examples could be investigated. We have pursued this approach for the Heisenberg group in \cite{borklin} and we expect it to be successful in other cases, possibly in a similar way to what has been done for the 3D contact case in \cite[Chapters 17 to 19]{agrachev2020} (and the references therein). In \cite{BBCN2017}, the authors classify the singularities of the sub-Riemannian exponential for low-dimensional generic structures.
    \item The study of conjugate points along abnormal geodesics is completely open.
\end{enumerate}                

\nocite{*}
\printbibliography
 
\end{document}